\documentclass{amsart}
\usepackage[mathscr]{eucal}
\usepackage{amssymb}
\usepackage{latexsym}
\usepackage{amsmath}
\newtheorem{theorem}{Theorem}
\newtheorem{corollary}{Corollary}[theorem]
\newtheorem{proposition}{Proposition}
\newtheorem{remark}{Remark}
\newtheorem{example}{Example}
\title{Notes on theta series for Niemeier lattices}
\author{Shoyu Nagaoka and Sho Takemori}
\begin{document}
\maketitle
\begin{abstract}
Some explicit expressions are given for the theta series of Niemeier lattices. As an application, 
we present some of their congruence relations.
%
%
\end{abstract}
\section{Introduction}
\label{intro}
The main object of this note is the theta series $\vartheta_{\mathcal{L}}^{(n)}$ 
associated with the Niemeier lattice $\mathcal{L}$:
\[
\vartheta_{\mathcal{L}}^{(n)}=\vartheta_S^{(n)}(Z)
:=\sum_{X\in M_{24,n}(\mathbb{Z})}{\rm exp}(\pi\sqrt{-1}{\rm tr}(S[X]Z)),\quad 
Z\in \mathbb{H}_n,
\]
where $S\in 2{\rm Sym}_{24}^*(\mathbb{Z})$ is the Gram matrix of $\mathcal{L}$, 
$S[X]:={}^t\!XSX$, and $\mathbb{H}_n$ is the Siegel upper half-space of degree $n$.
In the following, we will sometimes refer to $\vartheta_{\mathcal{L}}^{(n)}$ as the Niemeier theta
series.

A Niemeier lattice $\mathcal{L}$ is one of the 24 positive definite, even, unimodular
lattices of rank 24. Therefore, the Niemeier theta series $\vartheta_{\mathcal{L}}^{(n)}$
becomes a Siegel modular form of weight 12 and degree $n$. In this note, we give explicit
expressions for $\vartheta_{\mathcal{L}}^{(n)}$, using some modular forms of weight 12
with integral Fourier coefficients for the cases $n=2$ and 3.
\begin{theorem} 
\label{introtheorem}
Let $\mathcal{L}$ be a Niemeier lattice with Coxeter number
$h=h_{\mathcal{L}}$.
The theta series $\vartheta_{\mathcal{L}}^{(3)}$ has the following expression:
\begin{equation}
\label{main}
\begin{split}
\vartheta_{\mathcal{L}}^{(3)}
=& (E_4^{(3)})^3+(24h-720)Y_{12}^{(3)}+(48h^2-2880h+43200)X_{12}^{(3)}\\
 & +(48h^3-288h^2+3144h-1131120)F_{12},
\end{split}
\end{equation}
where $E_4^{(3)}$ is the Eisenstein series of weight 4 and degree 3; $Y_{12}^{(3)}$,
$X_{12}^{(3)}$ are Siegel modular forms of weight 12 with integral Fourier coefficients
defined in Section \ref{sec:4-1}; and $F_{12}$ is Miyawaki's cusp form of weight 12 (cf. \cite{Miyawaki}).
\end{theorem}
The expression (\ref{main}) leads to congruence relations among the Niemeier theta
series. For example, the congruence relations between the Coxeter numbers are related to those of the 
Niemeier theta series.
\begin{corollary}
\label{introcor1}
Let $\mathcal{L}_i$ $(i=1,2)$ be Niemeier lattices with Coxeter number $h_i:=h_{\mathcal{L}_i}$.
If $h_1 \equiv h_2 \pmod{m}$ for an integer $m$, then
\[
\vartheta_{\mathcal{L}_1}^{(3)} \equiv \vartheta_{\mathcal{L}_2}^{(3)} \pmod{m}.
\]
\end{corollary}
Conway and Sloane listed such
lattices according to the glue code and named them $\alpha$, $\beta$,\ldots, $\omega$
(cf. Conway and Sloane \cite{C-S}, p. 407, Table 16.1). 
\begin{theorem}
\label{introtheorem2}
Let $\alpha$, $\omega$, $\delta$, and $\psi$ be some of the Niemeier lattices defined above.
The following congruence relations hold:
\begin{align*}
& \vartheta_{\alpha}^{(3)} \equiv \vartheta_{\omega}^{(3)} \equiv \vartheta_{[4,2,6]}^{(3)} \pmod{23},\\
& \vartheta_{\delta}^{(3)} \equiv \vartheta_{\psi}^{(3)} \equiv \vartheta_{[2,2,12]}^{(3)} \pmod{23},
\end{align*}
where we use the following abbreviations: 
$[4,2,6]=\begin{pmatrix}4&1 \\ 1 & 6 \end{pmatrix}$, and 
$[2,2,12]=\begin{pmatrix}2&1 \\ 1 & 12\end{pmatrix}$ (see (\ref{abb}) in Section \ref{sec:4-2}).
\end{theorem}
These congruence relations lead to the following fact (also cf. \cite{Schulze}).
\begin{corollary}
\label{introcor1}
{\rm (1)} We have
\[
\varTheta(\vartheta_{\alpha}^{(2)}) \equiv \varTheta(\vartheta_{\omega}^{(2)}) \equiv
\varTheta(\vartheta_{\delta}^{(2)}) \equiv \varTheta(\vartheta_{\psi}^{(2)}) \equiv 0 \pmod{23},
\]
where $\varTheta$ is the theta operator (the generalized Ramanujan
operator) defined in Section \ref{sec:2-4} .\\
{\rm (2)} The Siegel modular forms
\[
\vartheta_{\alpha}^{(3)},\quad
\vartheta_{\omega}^{(3)},\quad
\vartheta_{\delta}^{(3)},\quad
\vartheta_{\psi}^{(3)}
\]
are the mod 23 singular modular forms with the maximal 23-rank 2 in the sense of Section \ref{sec:2-4}
(cf. \cite{B-K}).
\end{corollary}
\section{Preliminaries}
\label{sec:2}
\subsection{Notation}
\label{sec:2-1}
We begin by stating the notation that we will use. Let
$\Gamma_n:=Sp_n(\mathbb{Z})$ be the Siegel modular group of degree $n$, and let 
$\mathbb{H}_n$ be the Siegel upper half-space of degree $n$.
We denote by $M_k(\Gamma_n)$ the $\mathbb{C}$-vector space of all Siegel 
modular forms of weight $k$ for $\Gamma_n$, and denote by $S_k(\Gamma_n)$ the 
subspace of cusp forms.

Any $F(Z)$ in $M_k(\Gamma_n)$ has a Fourier expansion of the form
\[
F(Z)=
\sum_{0\leq T\in {\rm Sym}_n^*(\mathbb{Z})} a(F;T)q^T,\;
q^T:={\rm exp}(2\pi\sqrt{-1}{\rm tr}(TZ)),\;Z\in\mathbb{H}_n,
\]
where
\[
{\rm Sym}_n^*(\mathbb{Z}):=
\{\;T=(t_{ij})\in{\rm Sym}_n(\mathbb{Q})\;|\; t_{ii}, 2t_{ij}\in\mathbb{Z}\;\}.
\]
We will write the Fourier coefficient corresponding to 
$T\in {\rm Sym}_n^*(\mathbb{Z})$ as $a(F;T)$.

For a  subring $R$ of $\mathbb{C}$, let $M_k(\Gamma_n)_R\subset M_k(\Gamma_n)$ denote
the $R$-module of all modular forms whose Fourier coefficients lie in $R$.
\subsection{Formal $q$-expansion}
\label{sec:2-2}
For $T=(t_{ij})\in {\rm Sym}_n^*(\mathbb{Z})$ and $Z=(z_{ij})\in\mathbb{H}_n$,
we define $q_{ij}:={\rm exp}(2\pi\sqrt{-1}z_{ij})$. Then,
\[
q^T:={\rm exp}(2\pi\sqrt{-1}{\rm tr}(TZ))
=\prod_{i<j}q_{ij}^{2t_{ij}}\prod_{i=1}^n q_{ii}^{t_{ii}}.
\]
Therefore, we may consider $F\in M_k(\Gamma_n)_R$ as an element of the formal
power series ring
\[
F=\sum a(F;T)q^T\in R[q_{ij},q_{ij}^{-1}][\![q_{11},\ldots,q_{nn}]\!].
\]
For a prime number $p$, we denote by $\mathbb{Z}_{(p)}$ the local ring
of $p$-integral rational numbers. For two elements
\[
F_i=\sum a(F_i;T)q^T\in \mathbb{Z}_{(p)}[q_{ij},q_{ij}^{-1}][\![q_{11},\ldots,q_{nn}]\!]\;\;
(i=1,2),
\]
we write $F_1 \equiv F_2 \pmod{p}$ if the congruence relation
\[
a(F_1;T) \equiv a(F_2;T) \pmod{p}
\]
is satisfied for all $0\leq T\in {\rm Sym}_n^*(\mathbb{Z})$.
\subsection{Theta series for lattices and matrices}
\label{sec:2-3}
For a positive definite integral lattice $\mathcal{L}$ of rank $m$, we write the Gram matrix
as $S=S_{\mathcal{L}}\in{\rm Sym}_m(\mathbb{Z})$. We associate with it the theta series
\[
\vartheta_{\mathcal{L}}^{(n)}=\vartheta_S^{(n)}(Z)
:=\sum_{X\in M_{m,n}(\mathbb{Z})}{\rm exp}(\pi\sqrt{-1}{\rm tr}(S[X]Z)),\quad 
Z\in \mathbb{H}_n.
\]
In general, this becomes a Siegel modular form for some congruence subgroup of $\Gamma_n$.
In particular,
\[
\vartheta_{\mathcal{L}}^{(n)}=\vartheta_S^{(n)}(Z) \in M_{\frac{m}{2}}(\Gamma_n)_{\mathbb{Z}}
\]
if $\mathcal{L}$ is a positive definite, even, unimodular lattice of rank $m$.

For our use below, we now quote the following result, which is a special case of a theorem presented by 
B\"{o}cherer and Nagaoka (\cite{B-N}, Theorem 5).
\begin{theorem}
\label{ThB-N}
{\rm (\cite{B-N}, Theorem 5)}
Assume that $p\geq 2n+3$ and $p \equiv 3 \pmod{4}$. Let $S\in {\rm Sym}_2(\mathbb{Z})$ be a positive
definite binary quadratic form with det$(2S)=p$. Then, there exists a modular form
$G\in M_{\frac{p+1}{2}}(\Gamma_n)_{\mathbb{Z}_{(p)}}$ such that
\[
\vartheta_S^{(n)} \equiv G \pmod{p}.
\]
\end{theorem}
\begin{proof}
We apply Theorem 5 of \cite{B-N}, where
\[
f=\vartheta_S^{(n)}\in M_1(\Gamma_n,\chi_p)^0,\qquad
g=G\in M_{1+\frac{p-1}{2}}(\Gamma_n).
\]
\end{proof}
\subsection{Theta operator and mod $p$ singular modular form}
\label{sec:2-4}
First, we will define the theta operator, which is a differential operator. For 
$F=\sum a(F;T)q^T\in M_k(\Gamma_n)$, we associate with it the formal power series
\[
\varTheta (F):=\sum a(F;T)\cdot {\rm det}(T)q^T\in 
\mathbb{C}[q_{ij},q_{ij}^{-1}][\![q_{11},\ldots,q_{nn}]\!].
\]
This is called {\it the theta operator} (cf. \cite{B-N0}). For $n=1$, the classical theta operator
was studied by Ramanujan \cite{Sw}. It should be noted that
$\varTheta (F)$ is not necessarily of modular form.

Next, we introduce the mod $p$ singular modular form \cite{B-K}. 
For a prime number $p$, a modular form 
$F=\sum a(F;T)q^T\in M_k(\Gamma_n)_{\mathbb{Z}_{(p)}}$ is called 
{\it the mod $p$ singular modular form} with the (nontrivial) maximal $p$-rank $r$
($r<n$) if $F$ has the following property:
\[
a(F;T) \equiv 0 \pmod{p}
\]
for all $T\in {\rm Sym}_n^*(\mathbb{Z})$ with $r+1\leq {\rm rank}(T)\leq n$ and
\[
a(F;T)\not\equiv 0 \pmod{p}
\]
for some $T$ with rank$(T)=r$.

If $F=\sum a(F;T)q^T\in M_k(\Gamma_n)_{\mathbb{Z}_{(p)}}$ is a mod $p$ singular 
modular form,
then
\[
\varTheta (F) \equiv 0 \pmod{p}.
\]
Namely, $F$ is an element of the mod $p$ kernel of the theta operator.

In this note, we will show that the theta series associated with some Niemeier lattices
are examples of such forms.
\subsection{Sturm bound for Siegel modular forms}
\label{sec:2-5}
In this section, we introduce a result of Richter and Raum \cite{R-R} concerning the so-called
Sturm bound. From this result, we can specify a modular form by using mod $p$.
\begin{theorem}
\label{theorem4}
{\rm \cite{R-R}}
Assume that $p$ is a prime number and $F=\sum a(F;T)q^T$ is a modular form
in $M_k(\Gamma_n)_{\mathbb{Z}_{(p)}}$ $(n\geq 2)$. If
\[
a(F;T) \equiv 0 \pmod{p}
\]
for all $0\leq T=(t_{ij})\in {\rm Sym}_n^*(\mathbb{Z})$ with
\[
t_{ii}\leq \left(\frac{4}{3}\right)^n\frac{k}{16}, \qquad (i=1,\ldots,n),
\]
then
\[
a(F;T) \equiv 0 \pmod{p}
\]
for all $0\leq T\in {\rm Sym}_n^*(\mathbb{Z})$.
\end{theorem}
\begin{corollary}
\label{corollary2}
Let $F=\sum a(F;T)q^T$ be a modular form in $M_k(\Gamma_n)_{\mathbb{Z}_{(p)}}$.
If
\[
a(F;T)\in \mathbb{Z}
\]
for all $0\leq T\in {\rm Sym}_n^*(\mathbb{Z})$ with
\[
t_{ii}\leq \left(\frac{4}{3}\right)^n\frac{k}{16}, \qquad (i=1,\ldots,n),
\]
then $F\in M_k(\Gamma_n)_{\mathbb{Z}}$.
\end{corollary}
\subsection{Niemeier lattices}
\label{sec:2-6}
A Niemeier lattice is one of the 24 positive definite, even, unimodular lattices of
rank 24, which were classified by H.~Niemeier If $\mathcal{L}$ is a Niemeier
lattice, then the corresponding theta series $\vartheta_{\mathcal{L}}^{(n)}$
becomes a Siegel modular form of weight 12:
\[
\vartheta_{\mathcal{L}}^{(n)}\in M_{12}(\Gamma_n)_{\mathbb{Z}}.
\]
As stated in the Introduction, for the Niemeier lattices, we use the notation $\alpha$, $\beta$,\ldots, $\omega$, as defined by
Conway and Sloane (\cite{C-S}, p. 407, Table 16.1).
We write the associated theta series as
$\vartheta_{\alpha}^{(n)}$, $\vartheta_{\beta}^{(n)}$, $\ldots$. One of our
main purposes in this note is to study these theta series.
\section{Degree 2 theta series for Niemeier lattices}
\label{sec:3}
\subsection{Igusa's generators}
\label{sec:3-1}
Let $E_k^{(n)}$ be the Eisenstein series of degree $n$ and weight $k$, normalized 
as $a(E_k^{(n)};0_n)=1$. 
It is well known that $E_k^{(n)}\in M_k(\Gamma_n)_{\mathbb{Q}}$.

We set
\[
M(\Gamma_2)_{\mathbb{Z}}=\bigoplus_{k\in\mathbb{Z}} M_k(\Gamma_2)_{\mathbb{Z}}.
\]
Igusa \cite{Igusa} gave a minimal set of generators of the ring $M(\Gamma_2)_{\mathbb{Z}}$
over $\mathbb{Z}$. The set consists of 15 modular forms:
\[
M(\Gamma_2)_{\mathbb{Z}}=\mathbb{Z}[X_4,X_6,X_{10},X_{12},Y_{12},X_{16},\ldots, X_{48}].
\]
Here, the subscripts denote the weights, and
\[
X_4=E_4^{(2)},\quad X_6=E_6^{(2)},\quad X_{10}=\chi_{10},\quad X_{12}=\chi_{12},
\]
where $\chi_k\,(k=10,12)$ is Igusa's cusp form normalized as
\[
a\left(\chi_k;\begin{pmatrix}1&\tfrac{1}{2}\\ \tfrac{1}{2}&1\end{pmatrix}\right)=1.
\]
There are two modular forms of weight 12. The form $Y_{12}$ has the
$q$-expansion
\begin{align*}
Y_{12}=& (q_{11}+q_{22})-24(q_{11}^2+q_{22}^2)\\
       & +(q_{12}^{-2}+116q_{12}^{-1}+1206+116q_{12}+q_{12}^2)q_{11}q_{22}+\cdots
\end{align*}
and satisfies
\[
\varPhi(Y_{12})=\Delta,
\]
where $\varPhi$ is the Siegel operator, and
\begin{align*}
\Delta:&= \frac{1}{1728}\left( (E_4^{(1)})^3-(E_6^{(1)})^2 \right)\\
       &= q-24q^2+252q^3-1472q^4\cdots \in S_{12}(\Gamma_1)_{\mathbb{Z}}.
\end{align*}
For use below, we now give the expressions for $X_{12}^{(2)}=X_{12}$ and $Y_{12}^{(2)}=Y_{12}$
using the Eisenstein series:
\begin{equation}
\label{3-1-(1)}
\begin{split}
& X_{12}^{(2)} = a_1\cdot (E_4^{(2)})^3+a_2\cdot (E_6^{(2)})^2+a_3\cdot E_{12}^{(2)},
\vspace{1mm}
\\
&
a_1=\frac{131\cdot 593}{2^{11}\cdot 3^4\cdot 5^3\cdot 337},\;\;
               a_2=\frac{131\cdot 593}{2^{10}\cdot 3^6\cdot 7^2\cdot 337},\\
              & a_3=\frac{-131\cdot 593\cdot 691}{2^{11}\cdot 3^6\cdot 5^3\cdot 7^2\cdot 337},
\end{split}
\end{equation}
\begin{equation}
\label{3-1-(2)}
\begin{split}
& Y_{12}^{(2)} = b_1\cdot (E_4^{(2)})^3+b_2\cdot (E_6^{(2)})^2+b_3\cdot E_{12}^{(2)},\\
&
b_1=\frac{41\cdot 71\cdot 109}{2^7\cdot 3^3\cdot 5^3\cdot 337},\;\;
               b_2=\frac{1759}{2^2\cdot 3^4\cdot 7^2\cdot 337},\\
             &  b_3=\frac{-131\cdot 593\cdot 691}{2^7\cdot 3^4\cdot 5^3\cdot 7^2\cdot 337}.
\end{split}
\end{equation}
\subsection{Theta series for Niemeier lattices of degree 2}
\label{sec:3-2}
Let $\mathcal{L}$ be a Niemeier lattice. It is known that if the
Coxeter number of $\mathcal{L}$ is $h$, then
the $q$-expansion of $\vartheta_{\mathcal{L}}^{(1)}$ is given as follows:
\[
\vartheta_{\mathcal{L}}^{(1)}=1+24h\cdot q+\cdots .
\]
Since $\vartheta_{\mathcal{L}}^{(1)}\in\mathbb{Z}[E_4^{(1)},\Delta]$, the
form $\vartheta_{\mathcal{L}}^{(1)}$ can be expressed as
\begin{equation}
\label{equdegree1}
\vartheta_{\mathcal{L}}^{(1)}=(E_4^{(1)})^3+(24h-720)\Delta.
\end{equation}
This identity is a starting point for our study.
\begin{theorem}
\label{degree2}
Let $\mathcal{L}$ be a Niemeier lattice with Coxeter number $h$. Then
we have
\begin{equation}
\label{equdegree2}
\vartheta_{\mathcal{L}}^{(2)}=(E_4^{(2)})^3+(24h-720)Y_{12}^{(2)}
                              +(48h^2-2800h+43200)X_{12}^{(2)},
\end{equation}
where $X_{12}^{(2)}$, $Y_{12}^{(2)}$ are Igusa's generators, which were introduced in
\ref{sec:3-1}.
\end{theorem}
\begin{proof}
It follows from 
\[
\varPhi(\vartheta_{\mathcal{L}}^{(2)})=\vartheta_{\mathcal{L}}^{(1)},\quad
\varPhi(E_4^{(2)})=E_4^{(1)},\quad
\varPhi(Y_{12}^{(2)})=\Delta
\]
that $\vartheta_{\mathcal{L}}^{(2)}$ has the following expression:
\begin{equation}
\label{c1}
\vartheta_{\mathcal{L}}^{(2)}=(E_4^{(2)})^3+(24h-720)Y_{12}^{(2)}
                              +c_1\cdot X_{12},
\end{equation}
where $c_i=c_i(h)$ is a constant. We shall determine the value $c_1$.

We consider the diagonal restriction of both sides of (\ref{c1}).
We have
\begin{align*}
\vartheta_{\mathcal{L}}^{(2)}\left(\begin{pmatrix}z_{11}& 0\\ 0&z_{22}\end{pmatrix}\right)
 &=\vartheta_{\mathcal{L}}^{(1)}(z_{11})\cdot \vartheta_{\mathcal{L}}^{(1)}(z_{22})\\
 &=1+24h\cdot q_{11}+24h\cdot q_{22}+(24h)^2\cdot q_{11}q_{22}+\cdots,\\
(E_4^{(2)})^3\left(\begin{pmatrix}z_{11}& 0\\ 0&z_{22}\end{pmatrix}\right)
  & =(E_4^{(1)}(z_{11})\cdot E_4^{(1)}(z_{22}))^3\\
  & =1+720\cdot q_{11}+720\cdot q_{22}+720^2\cdot q_{11}q_{22}+\cdots,\\
Y_{12}^{(2)}\left(\begin{pmatrix}z_{11}& 0\\ 0&z_{22}\end{pmatrix}\right)
 &=E_4^{(1)}(z_{11})\Delta(z_{22})+E_4^{(1)}(z_{22})\Delta(z_{11})\\
 &=q_{11}+q_{22}+1440\cdot q_{11}q_{22}+\cdots,\\
X_{12}^{(2)}\left(\begin{pmatrix}z_{11}& 0\\ 0&z_{22}\end{pmatrix}\right)
 &=2^2\cdot 3\Delta(z_{11})\Delta(z_{22})\\
 &=12\cdot q_{11}q_{22}+\cdots,
\end{align*}
where $q_{ii}={\rm exp}(2\pi\sqrt{-1}z_{ii})$. Comparing the coefficients of
$q_{11}q_{22}$, we obtain the identity
\[
(24h)^2=720^2+1440(24h-720)+12\cdot c_1.
\]
This implies
\[
c_1=48h^2-2880h+43200.
\]
This completes the proof of Theorem \ref{degree2}.
\end{proof}
\section{Degree 3 theta series for Niemeier lattices}
\label{sec:4}
\subsection{Siegel modular forms of degree 3 and weight 12}
\label{sec:4-1}

Miyawaki \cite{Miyawaki} constructed a cusp form $F_{12}\in S_{12}(\Gamma_3)$
by using theta series with a spherical polynomial.
In this section, we introduce two modular forms of degree 3 and weight 12. We will then show
that they have integral Fourier coefficients.

We set
\begin{align*}
& X_{12}^{(3)}:=a_1\cdot (E_4^{(3)})^3+a_2\cdot (E_6^{(3)})^2+a_3\cdot E_{12}^{(3)}
                +\frac{4740}{337}F_{12},\\
& Y_{12}^{(3)}:=b_1\cdot (E_4^{(3)})^3+b_2\cdot (E_6^{(3)})^2+b_3\cdot E_{12}^{(3)}
                -\frac{356411}{337}F_{12},
\end{align*}
where $F_{12}\in S_{12}(\Gamma_3)_{\mathbb{Z}}$ is Miyawaki's cusp form \cite{Miyawaki}, 
and $a_i$ and $b_i$ are constants given in (\ref{3-1-(1)}), and (\ref{3-1-(2)}) in Section \ref{sec:3-1}.

We shall show that they have integral Fourier coefficients. Since $F_{12}$ are in cusp
form, we have
\[
\varPhi(X_{12}^{(3)})=X_{12}^{(2)},\qquad \varPhi(Y_{12}^{(3)})=Y_{12}^{(2)}.
\]
This means that, if rank$(T)<3$, then all of the Fourier coefficients $a(X_{12}^{(3)};T)$
and $a(Y_{12}^{(3)};T)$ are integral.

In the case that rank$(T)=3$, we have the following numerical data for
the Fourier coefficients of $X_{12}^{(3)}$, $Y_{12}^{(3)}$, and $F_{12}$:

\begin{table}[htbp]
\begin{center}
\begin{tabular}{c|ccc}

     $T$             &    $a(X_{12}^{(3)};T)$  &  $a(Y_{12}^{(3)};T)$  & $a(F_{12};T)$ \\ \hline
$[1,1,1;1,1,1]$      &          1            &       1       &       1     \\
$[1,1,1;0,0,1]$      &         84            &      7674     &      18     \\
$[1,1,1;0,0,0]$      &        1132           &     114476    &     164

\end{tabular}
\end{center}
\end{table}
Here, we used the abbreviation
\begin{equation}
\label{matrixdegree3}
[a,b,c;d,e,f]:=
\begin{pmatrix}
a          &    \tfrac{f}{2}         & \tfrac{e}{2} \\
\tfrac{f}{2}       &      b          & \tfrac{d}{2} \\
\tfrac{e}{2}       &    \tfrac{d}{2}       &   c
\end{pmatrix}
\in {\rm Sym}_3^*(\mathbb{Z}).
\end{equation}
\begin{proposition}
\label{prop}
The modular forms $X_{12}^{(3)}$, $Y_{12}^{(3)}$, and $F_{12}$ have integral
Fourier coefficients:
\[
X_{12}^{(3)},\;Y_{12}^{(3)}\in M_{12}(\Gamma_3)_{\mathbb{Z}},\quad
F_{12}\in S_{12}(\Gamma_3)_{\mathbb{Z}}.
\]
\end{proposition}
\begin{proof}
This is true for $F_{12}\in S_{12}(\Gamma_3)_{\mathbb{Z}}$ as a consequence of
its definition \cite{Miyawaki}.
We shall show that 
$X_{12}^{(3)}\in M_{12}(\Gamma_3)_{\mathbb{Z}}$. By Corollary \ref{corollary2},
it suffices to show that
\[
a(X_{12}^{(3)};T)\in\mathbb{Z}
\]
for all $0\leq T=(t_{ij})\in {\rm Sym}_3^*(\mathbb{Z})$ with $t_{ii}\leq 1$. This can be
confirmed from the above table. The same argument can be applied to 
$Y_{12}^{(3)}$.
\end{proof}
\subsection{Theta series for Niemeier lattices of degree 3}
\label{sec:4-2}
In this section, we show that $\vartheta_{\mathcal{L}}^{(3)}$ 
($\mathcal{L}:$ Niemeier lattice) can be expressed as an integral linear
combination of $(E_4^{(3)})^3$, $X_{12}^{(3)}$, $Y_{12}^{(3)}$, and
$F_{12}$. 
\begin{theorem}
\label{mainresult}
Let $\mathcal{L}$ be a Niemeier lattice with Coxeter number $h$.
Then we have
\begin{equation}
\label{degree3}
\begin{split}
\vartheta_{\mathcal{L}}^{(3)}= & 
(E_4^{(3)})^3+(24h-720)Y_{12}^{(3)}+(48h^2-2800h+43200)X_{12}^{(3)}\\
& +(48h^3-288h^2+3144h-1131120)F_{12},
\end{split}
\end{equation}
where $E_4^{(3)}$ is the Eisenstein series of degree 3 and weight 4;
and $X_{12}^{(3)}$, $Y_{12}^{(3)}$, and $F_{12}$ are modular forms
given in Section \ref{sec:4-1}. 
\end{theorem}
\begin{proof}
We note that
\[
\varPhi(E_4^{(3)})=E_4^{(2)},\quad
\varPhi(X_{12}^{(3)})=X_{12}^{(2)},\quad
\varPhi(Y_{12}^{(3)})=Y_{12}^{(2)}.
\]
Since $S_{12}(\Gamma_3)=\mathbb{C}\cdot F_{12}$, we can write 
\begin{equation*}
\label{4-2-(1)}
\begin{split}
\vartheta_{\mathcal{L}}^{(3)}= & 
(E_4^{(3)})^3+(24h-720)Y_{12}^{(3)}+(48h^2-2800h+43200)X_{12}^{(3)}\\
& +c_2\cdot F_{12}
\end{split}
\end{equation*}
for some constant $c_2$. By an argument similar to the one used in the proof of 
Theorem \ref{degree2}, we can express the value $c_2$ as a
polynomial in $h$. As in the case of degree 2, we consider the
diagonal restrictions:
\begin{align*}
\label{4-2-(2)}
& \vartheta_{\mathcal{L}}^{(3)}
\left(\begin{pmatrix}z_{11} & 0 & 0\\ 0 & z_{22} & 0\\ 0 & 0 & z_{33}
      \end{pmatrix}\right)
=\cdots +(24h)^3\cdot q_{11}q_{22}q_{33}+\cdots,\\
& \left( E_4^{(3)}
\left(\begin{pmatrix}z_{11} & 0 & 0\\ 0 & z_{22} & 0\\ 0 & 0 & z_{33}
      \end{pmatrix}\right)\right)^3
=\cdots +373248000\cdot q_{11}q_{22}q_{33}+\cdots,\\
& Y_{12}^{(3)}
\left(\begin{pmatrix}z_{11} & 0 & 0\\ 0 & z_{22} & 0\\ 0 & 0 & z_{33}
      \end{pmatrix}\right)
=\cdots+169632\cdot q_{11}q_{22}q_{33}+\cdots,\\
& X_{12}^{(3)}
\left(\begin{pmatrix}z_{11} & 0 & 0\\ 0 & z_{22} & 0\\ 0 & 0 & z_{33}
      \end{pmatrix}\right)
=\cdots+1728\cdot q_{11}q_{22}q_{33}+\cdots,\\
& F_{12}
\left(\begin{pmatrix}z_{11} & 0 & 0\\ 0 & z_{22} & 0\\ 0 & 0 & z_{33}
      \end{pmatrix}\right)
=\cdots+288\cdot q_{11}q_{22}q_{33}+\cdots,
\end{align*}
where $q_{ii}:={\rm exp}(2\pi\sqrt{-1}z_{ii})$. From these formulas,
we obtain
\begin{align*}
(24h)^3=373248000 &+169632\cdot (24h-720)\\
                 &+1728\cdot (48h^2-2880h+43200)\\
                 &+288\cdot c_2.
\end{align*}
This implies
\[
c_2=48h^3-288h^2+3144h-1131120.
\]
This completes the proof.
\end{proof}
\begin{remark}
\label{cpoly}
The polynomial coefficients in the expression of $\vartheta_{\mathcal{L}}^{(3)}$
are factored as follows:
\begin{equation}
\label{cpolynomial}
\begin{split}
& c_0(h):=24h-720=24(h-30),\\
& c_1(h):=48h^2-2880h+43200=48(h-30)^2,\\
& c_2(h):=48h^3-288h^2+3144h-1131120=24(h-30)(2h^2+48h+1571).
\end{split}
\end{equation}
All of these coefficients include the factor $h-30$.
If we take $\mathcal{L}=\gamma$ (cf. Conway-Sloane's list \cite{C-S}, p. 407,
Table 16.1), then the Coxeter number is just 30. In this case, we have
\[
\vartheta_{\gamma}^{(n)}=(E_4^{(n)})^3\qquad (n=1,2,3).
\]
This identity is also justified by the fact that
\[
\gamma=(E_8)^3,
\]
where $E_8$ is the so-called $E_8$ lattice.
\end{remark}
The following table is 
obtained from Theorem \ref{mainresult}. Here, $h=h_{\mathcal{L}}$
is the Coxeter number of the Niemeier lattice $\mathcal{L}$.
\newpage
\begin{table*}[hbtp]
\caption{Representation of degree three theta series of Niemeier lattices}
\begin{center}
\begin{tabular}{llll} \hline
Name      & Components  &  $h$   &  Theta series  \\ \hline
$\alpha$  & $D_{24}$           & 46 &  
$\vartheta_{\alpha}^{(3)}=(E_4^{(3)})^3+384Y_{12}^{(3)}
                                        +12288X_{12}^{(3)}+3076224F_{12}$ \\ \hline
$\beta$  & $D_{16}E_8$         & 30 &
$\vartheta_{\beta}^{(3)}=(E_4^{(3)})^3$ \\ \hline
$\gamma$  & $E_8^3$           &  30 &
$\vartheta_{\gamma}^{(3)}=\vartheta_{\beta}^{(3)}$ \\ \hline
$\delta$  & $A_{24}$           & 25 &
$\vartheta_{\delta}^{(3)}=(E_4^{(3)})^3-120Y_{12}^{(3)}
                                       +1200X_{12}^{(3)}-482520F_{12}$ 
\\ \hline
$\epsilon$  & $D_{12}^2$      & 22 &
$\vartheta_{\epsilon}^{(3)}=(E_4^{(3)})^3-192Y_{12}^{(3)}
                                       +3072X_{12}^{(3)}-690240F_{12}$ \\ \hline
$\zeta$  & $A_{17}E_7$        & 18 &
$\vartheta_{\zeta}^{(3)}=(E_4^{(3)})^3-288Y_{12}^{(3)}
                                       +6912X_{12}^{(3)}-887904F_{12}$ \\ \hline
$\eta$  & $D_{10}E_7^2$       & 18 &
$\vartheta_{\eta}^{(3)}=\vartheta_{\zeta}^{(3)}$\\ \hline
$\theta$  & $A_{15}D_9$       & 16 &
$\vartheta_{\theta}^{(3)}=(E_4^{(3)})^3-336Y_{12}^{(3)}
                                      +9408X_{12}^{(3)}-957936F_{12}$ \\ \hline
$\iota$  & $D_8^3$           & 14 &
$\vartheta_{\iota}^{(3)}=(E_4^{(3)})^3-384Y_{12}^{(3)}
                                     +12288X_{12}^{(3)}-1011840F_{12}$ \\ \hline
$\kappa$  & $A_{12}^2$       & 13 &
$\vartheta_{\kappa}^{(3)}=(E_4^{(3)})^3-408Y_{12}^{(3)}
                                     +13872X_{12}^{(3)}-1033464F_{12}$ \\ \hline

$\lambda$  & $A_{11}D_7E_6$  & 12 &
$\vartheta_{\lambda}^{(3)}=(E_4^{(3)})^3-432Y_{12}^{(3)}
                                     +15552X_{12}^{(3)}-1051920F_{12}$ \\ \hline
$\mu$  & $E_6^4$             & 12 &
$\vartheta_{\mu}^{(3)}=\vartheta_{\lambda}^{(3)}$ \\ \hline
$\nu$  & $A_{9}^2D_6$        & 10 &
$\vartheta_{\nu}^{(3)}=(E_4^{(3)})^3-480Y_{12}^{(3)}
                                    +19200X_{12}^{(3)}-1080480F_{12}$ \\ \hline
$\xi$  & $D_6^4$             & 10 &
$\vartheta_{\xi}^{(3)}=\vartheta_{\nu}$ \\ \hline
$o$  & $A_8^3$               & 9 &
$\vartheta_{o}^{(3)}=(E_4^{(3)})^3-504Y_{12}^{(3)}
                                    +21168X_{12}^{(3)}-1091160F_{12}$ \\ \hline
$\pi$  & $A_7^2D_5^2$        & 8 &
$\vartheta_{\pi}^{(3)}=(E_4^{(3)})^3-528Y_{12}^{(3)}
                                    +23232X_{12}^{(3)}-1099824F_{12}$ \\ \hline
$\rho$  & $A_6^4$            & 7 &
$\vartheta_{\rho}^{(3)}=(E_4^{(3)})^3-552Y_{12}^{(3)}
                                    +25392X_{12}^{(3)}-1106760F_{12}$ \\ \hline
                                    
$\sigma$  & $A_5^4D_4$       & 6 &
$\vartheta_{\sigma}^{(3)}=(E_4^{(3)})^3-576Y_{12}^{(3)}
                                    +27648X_{12}^{(3)}-1112256F_{12}$ \\ \hline
$\tau$  & $D_4^6$            & 6 &
$\vartheta_{\tau}^{(3)}=\vartheta_{\sigma}$ \\ \hline
$\upsilon$  & $A_4^6$        & 5 &
$\vartheta_{\upsilon}^{(3)}=(E_4^{(3)})^3-600Y_{12}^{(3)}
                                    +30000X_{12}^{(3)}-1116600F_{12}$ \\ \hline




$\phi$  & $A_3^8$            & 4 &
$\vartheta_{\phi}^{(3)}=(E_4^{(3)})^3-624Y_{12}^{(3)}
                                   +32448X_{12}^{(3)}-1120080F_{12}$ \\ \hline
$\chi$  & $A_2^{12}$         & 3 &
$\vartheta_{\chi}^{(3)}=(E_4^{(3)})^3-648Y_{12}^{(3)}
                                   +34992X_{12}^{(3)}-1122984F_{12}$ \\ \hline
$\psi$  & $A_1^{24}$         & 2 &
$\vartheta_{\psi}^{(3)}=(E_4^{(3)})^3-672Y_{12}^{(3)}
                                   +37632X_{12}^{(3)}-1125600F_{12}$ \\ \hline
$\omega$ & Leech             & 0 &
$\vartheta_{\omega}^{(3)}=(E_4^{(3)})^3-720Y_{12}^{(3)}
                                  +43200X_{12}^{(3)}-1131120F_{12}$ \\ \hline
\end{tabular}
\end{center}
\end{table*}
As stated in the Introduction, expression (\ref{degree3}) of Theorem \ref{mainresult}
is useful for studying the congruences between modular forms. 
As a straightforward conclusion, 
we can prove the following result.
\begin{corollary}
\label{coxetercong}
Let $\mathcal{L}_i$ $(i=1,2)$ be Niemeier lattices with Coxeter number $h_i:=h_{\mathcal{L}_i}$.
If $h_1 \equiv h_2 \pmod{m}$ for an integer $m$, then
\[
\vartheta_{\mathcal{L}_1}^{(3)} \equiv \vartheta_{\mathcal{L}_2}^{(3)} \pmod{m}.
\]
\end{corollary}
\begin{example}
Since $h_{\beta}=30$ and $h_{\rho}=7$ (see Table 1), we have
\[
\vartheta_{\beta}^{(3)} \equiv \vartheta_{\rho}^{(3)} \pmod{23}.
\]
For example,
\[
a(\vartheta_{\beta}^{(3)};[3,1,2])=749432632320 \equiv
799943308416=a(\vartheta_{\rho}^{(3)};[3,1,2]) \pmod{23}.
\]
Here we used the following abbreviation: 
For $\begin{pmatrix} a & \tfrac{b}{2} \\ \frac{b}{2}& c \end{pmatrix}
\in {\rm Sym}_2^*(\mathbb{Z})$, we set
\begin{equation}
\label{abb}
[a,b,c]
:=\begin{pmatrix} a & \tfrac{b}{2} \\ \tfrac{b}{2}& c \end{pmatrix}
\in {\rm Sym}_2^*(\mathbb{Z}).
\end{equation}
\end{example}
\begin{corollary}
\label{Lagrange}
Let $\{\,\mathcal{L}_i\;(i=1,2,3,4)\,\}$ be a set of Niemeier lattices with Coxeter number
$h_i=h_{\mathcal{L}_i}$ such that $h_1<h_2<h_3<h_4$.
Then, for any Niemeier lattice $\mathcal{L}$, 
the Niemeier theta series $\vartheta_{\mathcal{L}}^{(3)}$ has the following expression:
\[
\vartheta_{\mathcal{L}}^{(3)}=
\sum_{j=1}^4\ell_j(h_{\mathcal{L}})\vartheta_{\mathcal{L}_j}^{(3)},
\]
where $\ell_j(x)$\,$(j=1,2,3,4)$ are the Lagrange basis polynomials:
\[
\ell_j(x):=\prod_{\substack{1\leq m\leq 4\\ m\ne j}}\frac{x-x_m}{x_j-x_m}.
\]
\end{corollary}
\begin{proof}
We recall expression (\ref{degree3}) of Theorem \ref{mainresult}, and we solve the system of equations
\begin{equation}
\label{original}
\vartheta_{\mathcal{L}_i}^{(3)}
=(E_4^{(3)})^3+c_0(h_i)Y_{12}^{(3)}+c_1(h_i)X_{12}^{(3)}+c_2(h_i)F_{12},
\quad (i=1,2,3,4),
\end{equation}
with respect to $(E_4^{(3)})^3$, $Y_{12}^{(3)}$, $X_{12}^{(3)}$, and $F_{12}$.
Here, $c_j(h)$ is the polynomial defined in (\ref{cpolynomial}) of Remark \ref{cpoly}.
Since
\[
\begin{vmatrix}
1 & c_0(h_1) & c_1(h_1) & c_2(h_1) \\
1 & c_0(h_2) & c_1(h_2) & c_2(h_2) \\
1 & c_0(h_3) & c_1(h_3) & c_2(h_3) \\
1 & c_0(h_4) & c_1(h_4) & c_2(h_4) 
\end{vmatrix}
=55296\cdot \varDelta(h_1,h_2,h_3,h_4)\ne 0,\quad (\varDelta:\text{the differente}),
\]
the equations (\ref{original}) are solvable.
Again considering expression (\ref{degree3}), we conclude that 
$\vartheta_{\mathcal{L}}^{(3)}$ has the following expression:
\[
\vartheta_{\mathcal{L}}^{(3)}
=\sum_{j=1}^4 f_j(h_{\mathcal{L}})\vartheta_{\mathcal{L}_j}^{(3)},
\]
for some $f_j(x)\in \mathbb{Q}(h_1,h_2,h_3,h_4)[x]$. 
A direct calculation shows that
\[
f_j(x)=\prod_{\substack{1\leq m\leq 4\\ m\ne j}}\frac{x-x_m}{x_j-x_m}=\ell_j(x).
\]
This completes the proof of Corollary \ref{Lagrange}.
\end{proof}
\section{Congruence properties of theta series for Niemeier lattices}
\label{sec:5}
\subsection{Congruence relation between theta series}
\label{sec:5-1}
We will now prove some congruence relations satisfied by the Niemeier theta series.
For this, we need information about the Fourier coefficients of the generators $(E_4^{(3)})^3$,
$Y_{12}^{(3)}$, $X_{12}^{(3)}$, and $F_{12}$.

The results in \cite{O-Y} and \cite{Ka} can help us to calculate the Fourier coefficients of the 
Eisenstein series $E_k^{(3)}$.

By combining the Fourier coefficients of Miyawaki's cusp form $F_{12}$ (cf. \cite{Miyawaki}),
we obtain numerical examples of the Fourier coefficients of 
$\vartheta_{\mathcal{L}}^{(3)}$ ($\mathcal{L}$: Niemeier lattice).
We have the following result:
\begin{theorem}
\label{theoremcong}
The following congruence relations hold:
\begin{equation}
\label{cong}
\begin{split}
& \vartheta_\alpha^{(3)} \equiv
   \vartheta_\omega^{(3)} \equiv \vartheta_{[4,2,6]}^{(3)} \pmod{23},\\
& \vartheta_\delta^{(3)} \equiv
   \vartheta_\psi^{(3)} \equiv \vartheta_{[2,2,12]}^{(3)} \pmod{23},
\end{split}
\end{equation}
where $\alpha$, $\delta$, $\psi$, and $\omega$ are Niemeier lattices listed
in Table 1.
\end{theorem}
\begin{proof}
We prove the first congruence relation. 
Since det$([4,2,6])=23 \equiv 3 \pmod{4}$, we can apply 
Theorem \ref{ThB-N} to $\vartheta_{[4,2,6]}^{(3)}$. As a consequence,
there is a modular form $G_1\in M_{12}(\Gamma_3)_{\mathbb{Z}_{(23)}}$
such that
\[
\vartheta_{[4,2,6]}^{(3)} \equiv G_1 \pmod{23}.
\]
We obtain the following tables:
\begin{table}[htbp]
\begin{center}
\begin{tabular}{c|ccc}

   $T$     &    $a(\vartheta_{[4,2,6]}^{(2)};T)$  &  $a(\vartheta_\alpha^{(2)};T)$  
                                                  &  $a(\vartheta_\omega^{(2)};T)$ \\ \hline
$[0,0,0]$  &          1            &       1       &       1     \\
$[1,0,0]$  &          0            &     1104      &       0     \\
$[1,1,1]$  &          0            &    97152      &       0     \\
$[1,0,1]$  &          0            &   1022304     &       0

\end{tabular}
\end{center}
\end{table}
\begin{table}[htbp]
\begin{center}
\begin{tabular}{c|ccc}

   $T$     &    $a(\vartheta_{[4,2,6]}^{(3)};T)$  &  $a(\vartheta_\alpha^{(3)};T)$  
                                                  &  $a(\vartheta_\omega^{(3)};T)$ \\ \hline
$[1,1,1;1,1,1]$  &          0            &   4177536       &       0     \\
$[1,1,1;0,0,1]$  &          0            &   81607680      &       0     \\
$[1,1,1;0,0,0]$  &          0            &   781393536     &       0     \\

\end{tabular}
\end{center}
\end{table}
\\
Here, we used the abbreviation $[a,b,c;d,e,f]$ introduced in Section 4.1, (\ref{matrixdegree3}).
From the information in the above tables, we can show that
\[
a(\vartheta_{[4,2,6]}^{(3)};T) \equiv a(G_1;T)
                               \equiv a(\vartheta_\alpha^{(3)};T)
                               \equiv a(\vartheta_\omega^{(3)};T) \pmod{23}
\]
for all $T=(t_{ij})\in {\rm Sym}_3^*(\mathbb{Z})$ with $t_{ii}\leq 1$. By using
Theorem \ref{theorem4}, we obtain
\[
\vartheta_{[4,2,6]}^{(3)} \equiv \vartheta_\alpha^{(3)} \equiv \vartheta_\omega^{(3)}
\pmod{23}.
\]
The proof of the second congruence relation proceeds in a similar manner.
There is a modular form $G_2\in M_{12}(\Gamma_3)_{\mathbb{Z}_{(23)}}$ such that
\[
\vartheta_{[2,2,12]}^{(3)} \equiv G_2 \pmod{23}.
\] 
In this case, we obtain the following tables:
\begin{table}[htbp]
\begin{center}
\begin{tabular}{c|ccc}

   $T$     &    $a(\vartheta_{[2,2,12]}^{(2)};T)$  &  $a(\vartheta_\delta^{(2)};T)$  
                                                  &  $a(\vartheta_\psi^{(2)};T)$ \\ \hline
$[0,0,0]$  &          1            &       1       &       1     \\
$[1,0,0]$  &          2            &     600       &      48     \\
$[1,1,1]$  &          0            &    27600      &       0     \\
$[1,0,1]$  &          0            &   303600      &      2208

\end{tabular}
\end{center}
\end{table}
{}
\begin{table}[htbp]
\begin{center}
\begin{tabular}{c|ccc}

   $T$     &    $a(\vartheta_{[2,2,12]}^{(3)};T)$  &  $a(\vartheta_\delta^{(3)};T)$  
                                                  &  $a(\vartheta_\psi^{(3)};T)$ \\ \hline
$[1,1,1;1,1,1]$  &          0            &   607200        &       0     \\
$[1,1,1;0,0,1]$  &          0            &  12751200       &       0     \\
$[1,1,1;0,0,0]$  &          0            &  127512000      &    97152    \\

\end{tabular}
\end{center}
\end{table}

Again, by Theorem \ref{theorem4}, we obtain
\[
\vartheta_{[2,2,12]}^{(3)} \equiv G_2 \equiv \vartheta_\delta^{(3)} \equiv \vartheta_\psi^{(3)}
\pmod{23}.
\]
This completes the proof.
\end{proof}
\begin{remark}
The congruence relations
\[
\vartheta_{\alpha}^{(3)} \equiv \vartheta_{\omega}^{(3)} \pmod{23}\;\;\text{and}\;\;
\vartheta_{\delta}^{(3)} \equiv \vartheta_{\psi}^{(3)} \pmod{23}
\]
can be also proved by Corollary \ref{coxetercong}.
\end{remark}
\subsection{Theta operator on theta series and the mod $p$ singular form}
\label{sec:5-2}
In the previous section, we saw some congruence relations arising from
theta series. Such relations can be reformulated by the terminology of the
theta operator and the mod $p$ singular forms that were introduced in Section \ref{sec:2-4}.
\begin{theorem}
{\rm (1)} The following congruence relations hold:
\[
\varTheta (\vartheta_\alpha^{(2)}) \equiv
\varTheta (\vartheta_\delta^{(2)}) \equiv
\varTheta (\vartheta_\psi^{(2)}) \equiv
\varTheta (\vartheta_\omega^{(2)}) \equiv
0 \pmod{23}.
\]
{\rm (2)} The theta series
$\vartheta_\alpha^{(3)}$,
$\vartheta_\delta^{(3)}$,
$\vartheta_\psi^{(3)}$, and
$\vartheta_\omega^{(3)}$
are mod 23 singular modular forms with the nontrivial maximal 23-rank 2.
\end{theorem}
\begin{proof}
Statement (1) is a consequence of Theorem \ref{theoremcong}. To prove 
statement (2),
we must show that
\[
a(\vartheta_{\mathcal{L}}^{(3)};T) \not\equiv 0 \pmod{23}
\]
for some $T$ with rank$(T)=2$. This comes from the following:
\begin{align*}
& a(\vartheta_\alpha^{(2)};[2,1,3]) \equiv a(\vartheta_\delta^{(2)};[1,1,6])\\
& \equiv a(\vartheta_\psi^{(2)};[1,1,6]) \equiv a(\vartheta_\omega^{(2)};[2,1,3])\\
& \equiv 2 \pmod{23}.
\end{align*}
\end{proof}
\noindent
\textbf{Acknowledgement:}\;
The second author is partially supported
by the Grants-in-aid (S) (No. 23224001).



Shoyu Nagaoka\\
Department of Mathematics\\
Kindai University\\
Higashi-Osaka, Osaka 577-8502, Japan\\
Email:nagaoka@math.kindai.ac.jp\\
and\\
Sho Takemori\\
Department of Mathematics\\
Hokkaido University\\
Kita 10, Nishi 8, Kita-Ku\\
Sapporo, Hokkaido, 060-0810, Japan\\
E-mail: stakemorii@gmail.com

\end{document}